\documentclass{amsart}

\usepackage{amscd,amsxtra,amsthm}
\usepackage[all]{xy}
\usepackage{etex}
\usepackage{pictex}
\usepackage{graphicx}
\usepackage{mathtools}
\usepackage{float}
\usepackage{xcolor}
\usepackage[utf8]{inputenc}
\usepackage[T1]{fontenc}
 \usepackage{amssymb}
\usepackage{tikz}
\usepackage{tkz-tab}
\usepackage{caption}
\usepackage{subcaption}

\usepackage{multirow}
\newcommand{\ceil}[1]{\left \lceil #1 \right \rceil}
\newcommand{\floor}[1]{\left \lfloor #1 \right \rfloor}
\usepackage{comment}

\newtheorem{theorem}{Theorem}
\newtheorem{lemma}[theorem]{Lemma}

\theoremstyle{definition}
\newtheorem{definition}[theorem]{Definition}
\newtheorem{example}[theorem]{Example}

\theoremstyle{corollary}
\newtheorem{corollary}[theorem]{Corollary}

\theoremstyle{conjecture}

\theoremstyle{remark}

\theoremstyle{proposition}

\begin{document}
\parskip0pt
\parindent15pt
\baselineskip15pt    

\title{Lower Bounds for Multicolor Star-Critical Ramsey Numbers} 

\author[M. Budden]{Mark Budden}
\address{Department of Mathematics and Computer Science \\
Western Carolina University \\
Cullowhee, NC 28723 USA}
\email{mrbudden@email.wcu.edu}

\author[Y.S. Khobagade]{Yash Shamsundar Khobragade}
\address{Department of Mathematics \\ Indian Institute of Science Education and Research Bhopal \\ Bhopal 462066 India}
\email{yashk19@iiserb.ac.in}

\author[S. Sarkar]{Siddhartha Sarkar}
\address{Department of Mathematics \\ Indian Institute of Science Education and Research Bhopal \\ Bhopal 462066 India}
\email{sidhu@iiserb.ac.in}

\subjclass[2020]{Primary 05C55, 05D10; Secondary 05C40}
\keywords{multicolor Ramsey numbers, edge colorings, paths}

\begin{abstract} The star-critical Ramsey number is a refinement of the concept of a Ramsey number. In this paper, we give equivalent criteria for which the star-critical Ramsey number vanishes.  Next, we provide a new general lower bound for multicolor star-critical Ramsey numbers whenever it does not vanish.  As an application, we evaluate $r_*(P_k, P_3, P_3)$, where $P_n$ is a path of order $n$.  In the process of proving these results, we also show that $r_*(C_5, P_3)=3$, where $C_5$ is a cycle of order $5$.  \end{abstract}

\maketitle

\section{Introduction}\label{intro}

We assume that all graphs are finite and simple in that they do not contain any loops or multiedges. For positive integers $s$ and $t$ such that $s \leq t$, we write $[s,t] := \{ i  \in {\mathbb{N}} \ | \  s \leq i \leq t \}$. The notations $K_n$, $P_n$, and $C_n$ represent the complete graph, the path, and the cycle of order $n$, respectively.  A {\it $t$-coloring} of a graph $G=(V(G),E(G))$ is a map $$f:E(G)\longrightarrow [1,t].$$  Such an edge coloring is not assumed to be surjective nor is it assumed to be proper (i.e., adjacent edges may receive the same color).
For graphs $G_1, G_2, \dots , G_t$, the {\it Ramsey number} $r(G_1, G_2, \dots , G_t)$ is the least natural number $p$ such that every $t$-coloring of the edges of $K_p$ contains a subgraph that is isomorphic to $G_i$ in color $i$, for some $i\in [1,t]$.  A {\it critical coloring} for $r(G_1, G_2, \dots , G_t)$ is a $t$-coloring of $K_{r(G_1, G_2, \dots , G_t)-1}$ that avoids a copy of $G_i$ in color $i$, for all $i\in [1,t]$.  The current known values of multicolor Ramsey numbers can be found in Radziszowski's dynamic survey \cite{Rad}.

In 2010, Jonelle Hook \cite{Hook} introduced a refinement of the Ramsey number known as a star-critical Ramsey number.  In order to define it, we must first introduce the notation $K_n\sqcup K_{1,k}$ to be the graph formed by taking $K_n$, adding in a new vertex, then joining that vertex to exactly $k$ vertices in the $K_n$.  The {\it star-critical Ramsey number} $r_*(G_1, G_2, \dots , G_t)$ is then defined to be the least $k$ such that every $t$-coloring of $K_{r(G_1, G_2, \dots , G_t)-1}\sqcup K_{1,k}$ contains a subgraph that is isomorphic to $G_i$ in color $i$, for some $i\in [1,t]$.
Note that \begin{equation} 1\le r_*(G_1, G_2, \dots , G_t)\le r(G_1, G_2, \dots , G_t)-1,\label{starcritgen} \end{equation} if the graphs $G_1, G_2, \dotsc, G_t$ are connected and of order at least $2.$  In Section \ref{star-crit-zero}, we will completely determine the scenarios when $r_{\ast}(G_1, G_2, \dots , G_t)$ vanishes. In the special case where $t=1$, we note that $r(G)=|V(G)|$ and $r_*(G)=\delta (G)$, where $\delta (G)$ is the {\it minimum degree} of $G$: $$\delta (G):=\min \{{\rm deg}_G(v) \ | \ v\in V(G)\}.$$

A {\it proper vertex coloring} of a graph $G$ is a map $$c:V(G)\longrightarrow \{1, 2, \dots \ell\}$$ such that $c(u)\ne c(v)$ when $uv\in E(G)$.  The {\it chromatic number} of $G$, denoted $\chi (G)$ is the least $\ell$ for which such a coloring exists.  The {\it chromatic surplus} of $G$, denoted $s(G)$ is the smallest order of a color class among all proper vertex colorings of $G$ that use exactly $\chi (G)$ colors. Now let ${\mathcal{C}}$ denote the collection of all proper vertex colorings of $G$ that admit a color class of size $s(G)$. For any $c \in {\mathcal{C}}$, let $V_1, V_2, \dots , V_{\chi(G)}$ be the distinct color classes in the coloring $c$, where $$s(G)=|V_1|\le |V_2| \le \cdots \le |V_{\chi(G)}|.$$ If $v\in V_1$, let ${\rm deg}_{V_i}(v)$ be the number of edges that join $v$ to $V_i$, where $2\le i\le \chi (G)$. Then define $$\tau_{c}(G):=\min \{ {\rm deg}_{V_i}(v) \  | \ v\in V_1 \ \mbox{and} \ 2\le i\le \chi (G)\}.$$ and  $$\tau(G) := \min \{ \tau_{c}(G) \ | \ c\in \mathcal{C}\}.$$   

A graph is called {\it connected} if there exists a path joining every distinct pair of vertices.  If a graph is not connected, it is called {\it disconnected}, and the maximal connected subgraphs are called its {\it connected components}.  A connected graph has one connected component.  The {\it vertex connectivity} of a graph $G$, denoted $\kappa (G)$, is the minimum number of vertices whose deletion results in a disconnected graph or a single vertex.  

In 1981, Burr \cite{Burr} gave an improvement on Chv\'atal and Harary's \cite{CH3} lower bound for $2$-color Ramsey numbers.  To state Burr's bound, let $c(G_1)$ be the order of a largest connected component of $G_1$. If $c(G_1) \geq s(G_2)$, then $$r(G_1, G_2)\ge (c(G_1)-1)(\chi (G_2)-1)+s(G_2),$$ and with this hypothesis a graph $G_1$ is called {\it $G_2$-good} if the equality holds. More generally, if $r(G_1, G_2, \dotsc, G_{t-1}) \geq s(G_t)$, then
\begin{eqnarray}\label{burr-gen-ineq}
r(G_1, G_2, \dots , G_t)\ge (r(G_1, G_2, \dots , G_{t-1})-1)(\chi (G_t)-1)+s(G_t),
\end{eqnarray} 
and $(G_1, G_2, \dots , G_{t-1})$ is called {\it $G_t$-good} if equality holds.

This paper is organized as follows. 
In Section \ref{star-crit-zero} we provide equivalent criteria for which the star-critical Ramsey number vanishes (Theorem \ref{star-crit-zero-thm}). In Section \ref{genlower}, we provide an overview of the current known lower bounds for star-critical Ramsey numbers (proved in \cite{BD}, \cite{HL1}, \cite{HL2}, and \cite{ZBC}) and we offer an improved multicolor lower bound (Theorem \ref{lower}), as was initiated in \cite{Kho}).  To demonstrate the utility of this new lower bound, we turn our attention in Section \ref{3paths} to the following evaluation (Theorem \ref{k33theorem}):
$$r_*(P_k, P_3, P_3)=\left\{ \begin{array}{ll} 1 & \mbox{if $k=3$} \\ 3 & \mbox{if $k=4$} \\ 4 & \mbox{if $k=5$} \\ 3 & \mbox{if $k\ge 6$.}\end{array}\right.$$  In the process of proving this result on paths, we also show that $r_*(C_5, P_3)=3$.

\section{Star-Critical Ramsey Number Zero: An Exceptional Case}\label{star-crit-zero}

For a collection of graphs $G_1, G_2, \dotsc, G_t ~(t \geq 2)$ of order at least $2$, we will examine a situation where the star-critical Ramsey number is $0$, as opposed to Inequality (\ref{starcritgen}) in Section \ref{intro}, where the graphs were assumed to be connected. We must first introduce some notations and terminologies.

A graph $G$ is called {\it discrete} if and only if $\chi(G) = 1$, and it contains an isolated vertex if and only if $\delta(G) = 0$. If $G$ has order at least $2$ we define 
\[
G^{\prime} := 
\begin{cases}
G - w &\mbox{if } \delta(G) = 0 {\mathrm{~and~}} {\mathrm{deg}}_{G}(w) = 0 \\
G &\mbox{if } \delta(G) \geq 1,
\end{cases}
\] 
where ${\mathrm{deg}}_{G}(w)$ denotes the degree of the vertex $w$ in the graph $G$.  Here, $G-w$ is the subgraph of $G$ induced by $V(G)-\{w\}$.
Note that $G^{\prime}$ is well-defined up to the isomorphism of graphs.

Let ${\mathcal{G}}$ denote the multiset ${\mathcal{G}} := \{ G_1, \dotsc, G_t \}$. In the following, we will use the notations $r({\mathcal{G}})$ and $r_*(\mathcal{G})$ to denote the Ramsey number $r(G_1, \dotsc, G_t)$  and star-critical Ramsey number $r_{\ast}(G_1, \dotsc, G_t)$, respectively. Let $I_{00} \subseteq I_{0} \subseteq [1,t]$ be the subsets defined by
\begin{eqnarray*}
I_0 & := & \{ i \in [1,t] \ | \ \delta(G_i) = 0 \}, \\
I_{00} & := & \{ i \in I_0 \ | \  |V(G_i)| = r({\mathcal{G}}) \}.
\end{eqnarray*}

\noindent Next, for a subset $J \subseteq [1,t]$, we define the multiset ${\mathcal{G}}^{\prime}_J = \{ F_1, \dotsc, F_t \}$ given by
\[
F_i := 
\begin{cases}
G^{\prime}_i &\mbox{if } i \in J, \\
G_i &\mbox{if } i \in [1,t]- J.
\end{cases}
\]  
The following lemma shows that removing a single isolated vertex from some of the graphs in ${\mathcal{G}}$ might only change the Ramsey number by $1$.

\begin{lemma}\label{bounding-lem}
Let $t \geq 2$ be an integer, and ${\mathcal{G}} = \{ G_i \ | \ i \in [1,t] \}$ be a multiset of simple graphs of order at least $2$. Then for any subset $J \subseteq [1,t]$, we have $r({\mathcal{G}}) - 1 \leq r({\mathcal{G}}^{\prime}_J) \leq r({\mathcal{G}})$.
\end{lemma}

\begin{proof}
Without loss of generality, we may assume that $I_{0} \neq \emptyset$, $J \subseteq I_{0}$ and $\emptyset \neq J = [1,k]$ for some $1 \leq k \leq t$. Since $G^{\prime}_j$ is a subgraph of $G_j$ for each $j \in J$, the second inequality is immediate. We set $n := r({\mathcal{G}})$, and assume that $r({\mathcal{G}}^{\prime}_J) = n - r < n$ for some $n-r \geq 1$. Then, we have $1 \leq r \leq n-1$. It is enough to show that $r=1$. 

Assume that $r \geq 2$. Then, we have $n-r < n-r+1 < n$. We will show that any $t$-coloring of $K_{n-r+1}$ contains a monochromatic copy of $G_i$ is color $i$, for some $i \in [1,t]$, which is a contradiction. 
Let ${\mathcal{H}}$ be a $t$-coloring of $K_{n-r+1}$ that avoids a copy of $G_i$ in color $i$ for all $i \in [1,t] - J$. Let $v$ be any vertex in $K_{n-r+1}$, and consider the restriction ${\overline{\mathcal{H}}}$ of ${\mathcal{H}}$ to the subgraph $K_{n-r+1} - v \cong K_{n-r}$. Since $r({\mathcal{G}}^{\prime}_J) = n - r$, it follows that $K_{n-r+1} - v$ contains a copy of $G^{\prime}_j$ in color $j$, say $K$ in the $t$-coloring ${\overline{\mathcal{H}}}$ for some $j \in J$. Then the disjoint union $K \cup v$ is a subgraph of $K_{n-r+1}$, which is a monochromatic copy of $G_j$ in color $j$ in the $t$-coloring ${\mathcal{H}}$. This contradicts the inequality $n-r+1 < r({\mathcal{G}})$, from which it follows that $r=1$.
\end{proof}

In the following two examples, observe that both of the bounds for the Ramsey number $r({\mathcal{G}}^{\prime}_J)$ given in the above lemma can occur.  Our examples include the known Ramsey numbers $r(P_3, K_{1,3})=5$ \cite{Par} and $r(P_3, P_4)=4$ \cite{GG}.

\begin{example}
Let $G_1 = P_3$, $G_2 = K_{1,3} \cup w$, and $G^{\prime}_2 = G_2 - w$ (see $(i)$, $(ii)$, and $(iii)$ in Figure \ref{ex1}) and we claim that $r(G_1, G_2) = r(G_1, G^{\prime}_2) = 5$.  
\begin{figure}[h]
   \centering
      \includegraphics[width=0.97\textwidth]{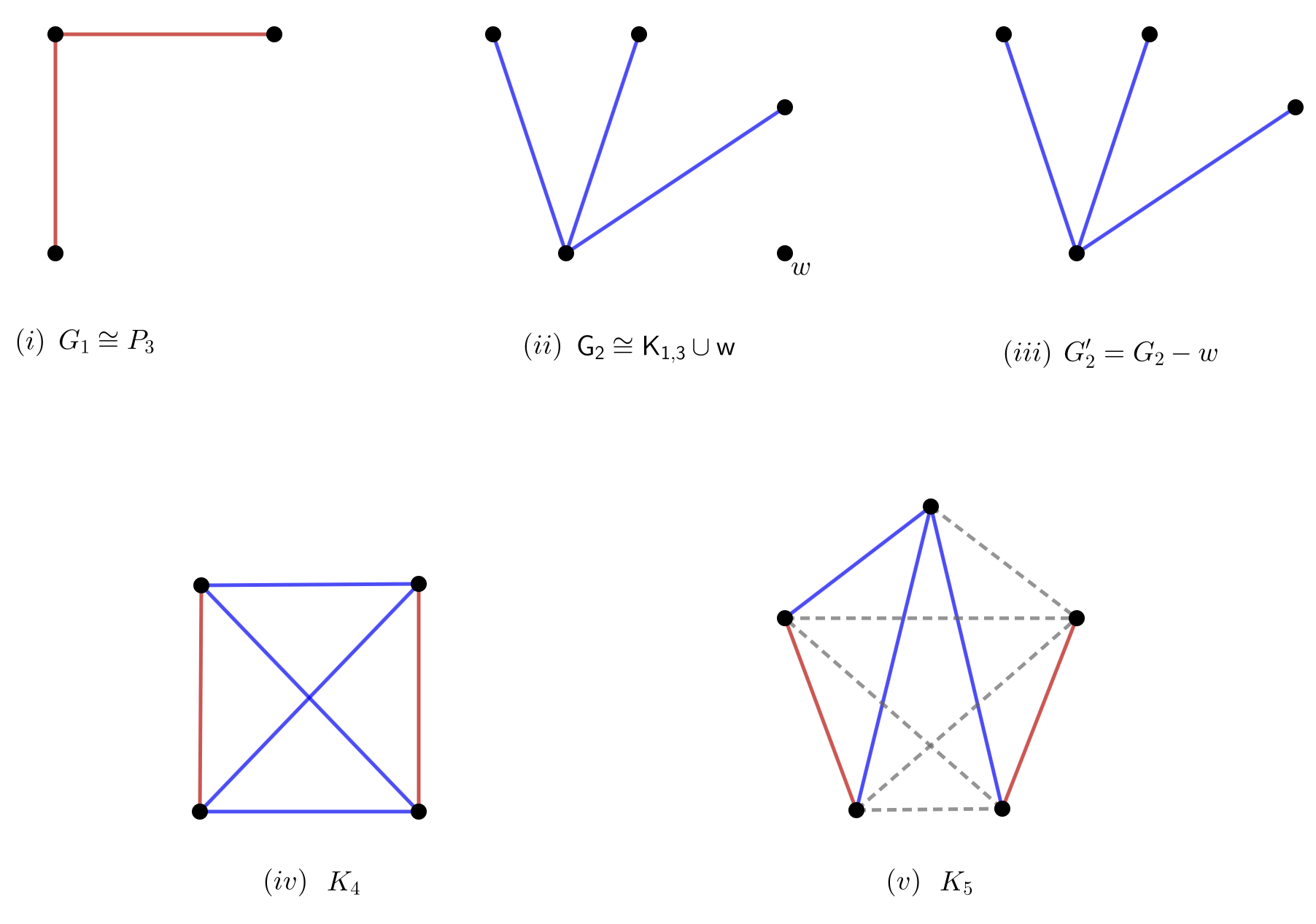}
   \caption{Showing that $r(P_3, K_{1,3}\cup w)=r(P_3, K_{1,3})=5$.}\label{ex1}
\end{figure}
The $2$-coloring of $K_4$ shown in $(iv)$ of Figure \ref{ex1} avoids a red $G_1$ and a blue $G_2$ (and hence, a blue $G^{\prime}_2$). This implies that $r(G_1, G_2)\ge 5$ and $r(G_1, G^{\prime}_2) \geq 5$. 
To verify $r(G_1, G_2) = r(G_1, G^{\prime}_2) = 5$, consider a red-blue coloring of $K_5$ that avoids a red copy of $G_1$. Then there are at most two disjoint red edges, leaving a vertex $v$ which is joined to the remaining vertices by blue edges (see $(v)$ in Figure \ref{ex1}).  The result is a $2$-coloring that contains a blue $G_2$ (and hence, a blue $G_2'$).  It follows that $r(G_1, G_2)=r(G_1, G_2')=5$.
\end{example}

\begin{example} 
Consider $G_1 = P_3, G_2 = P_4 \cup w$, and $G^{\prime}_2 = G_2 - w$ (see $(i)$, $(ii)$, and $(iii)$ in Figure \ref{ex2}), and we claim that $r(G_1,G^{\prime}_2) = 4$, whereas $r(G_1, G_2) = 5$. 
\begin{figure}[h]
  \centering
    \includegraphics[width=0.96\textwidth]{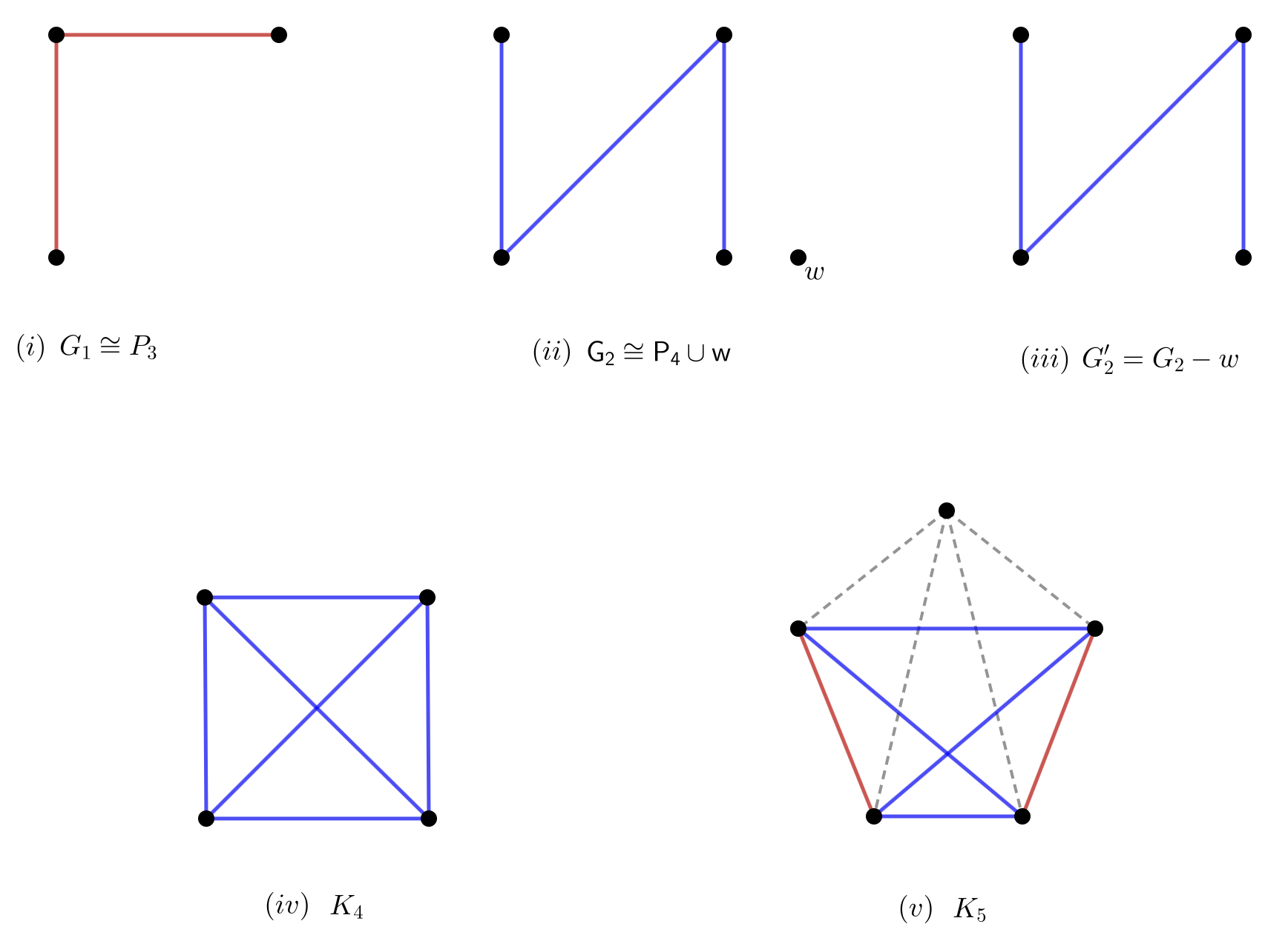}
  \caption{The $2$-colorings corresponding to $r(P_4, P_4 \cup w)=5$ and $r(P_3, P_4)=4$.}\label{ex2}
\end{figure}
Recall that $r(G_1,G^{\prime}_2) = r(P_3, P_4) = 4$. Since, $|V(G_2)| = 5$ and a blue $K_4$ does not contain a red $G_1$, we have $r(G_1, G_2) \geq 5$ (see $(iv)$ in Figure \ref{ex2}). 
Now we will show that $r(G_1, G_2) = 5$. We consider a red-blue coloring of $K_5$ that avoids a red copy of $G_1$. If a red $P_3$ is avoided, then there are at most two disjoint red edges and the remaining edges are blue. Then it contains a blue $P_4$ which is extendible to a copy of a blue $G_2$ (see $(v)$ in Figure \ref{ex2}).
\end{example}

Now we are ready to prove a theorem which provides equivalent criteria for when $r_{\ast}({\mathcal{G}})$ attains the value $0$.

\begin{theorem}\label{star-crit-zero-thm}
Let $t \geq 2$ be an integer and ${\mathcal{G}} = \{ G_i \ | \  i \in [1,t] \}$ be a multiset of simple graphs of order at least $2$. Then $r_{\ast}({\mathcal{G}}) = 0$ if and only if one of the following statements hold:
\begin{enumerate}
\item[(i)] $G_j$ is discrete for some $j \in [1,t]$, 
\item[(ii)] $G_1, \dotsc, G_t$ are non-discrete graphs and $r({\mathcal{G}}_{I_{00}}) = r({\mathcal{G}}) - 1$.
\end{enumerate}
\end{theorem} 

\begin{proof}
($\Longleftarrow$) First assume that (i) is true. We claim that 
\begin{equation}\label{star-crit-zero-ineq1}
r({\mathcal{G}}) = \min \big\{ |V(G_i)| \ | \  G_i \mbox{ is discrete and }  i \in [1,t] \big\}.
\end{equation} 
Let $j_0 \in [1,t]$ be chosen such that $G_{j_0}$ is discrete, and 
\[
|V(G_{j_0})| = \min \big\{ |V(G_i)| \ | \ G_i  \mbox{ is discrete and }  i \in [1,t] \big\}.
\] 
Since a $K_{|V(G_{j_0})| - 1}$ with all edges colored $i$ is a $t$-coloring which avoids a copy $G_i$ in color $i$, for all $i \in [1,t]$, it follows that $r({\mathcal{G}}) \geq |V(G_{j_0})|$. 
Next, for any integer $k \geq |V(G_{j_0})|$, every $t$-coloring of $K_k$ contains a copy of the discrete graph $G_{j_0}$ in color $j_0$. This implies that $r({\mathcal{G}}) \leq |V(G_{j_0})|$, and Equation (\ref{star-crit-zero-ineq1}) follows. 

Let $i_0 \in [1,t]$ with $G_{i_0}$ discrete so that $r({\mathcal{G}}) = |V(G_{i_0})|$. Let ${\mathcal{H}}$ be an arbitrary critical coloring of $K_{|V(G_{i_0})| - 1}$; that is, $\mathcal{H}$ is a $t$-coloring of $K_{|V(G_{i_0})| - 1}$ that avoids a copy of $G_i$ in color $i$, for all $i \in [1,t]$. Introduce a new vertex $v$ and consider the disjoint union ${\mathcal{H}} \cup v$. Since $G_{i_0}$ is discrete with $|V(G_{i_0})|$ many vertices, ${\mathcal{H}} \cup v$ contains a copy of $G_{i_0}$ in color $i_0$. Since ${\mathcal{H}}$ was arbitrary, we have $r_{\ast}({\mathcal{G}}) = 0$.  

Now we show that (ii) implies $r_{\ast}({\mathcal{G}}) = 0$. Set $n := r({\mathcal{G}})$, and consider a critical coloring ${\mathcal{H}}$ for $r({\mathcal{G}})$; that is, $\mathcal{H}$ is a $t$-coloring of $K_{n-1}$ which avoids a copy of $G_i$ in color $i$, for every $i \in [1,t]$. Since $n-1 = r({\mathcal{G}}_{I_{00}})$, ${\mathcal{H}}$ contains a copy of $G^{\prime}_{j_0}$ for some $j_0 \in I_{00}$. By definition of $I_{00}$ and ${\mathcal{G}}_{I_{00}}$, $G_{j_0}$ contains an isolated vertex $w$ and $|V(G_{j_0})| = r({\mathcal{G}})$. Introduce a new isolated vertex $v$ to the critical coloring ${\mathcal{H}}$. Then this isolated vertex $v$ along with the copy of $G^{\prime}_{j_0}$ in ${\mathcal{H}}$ is isomorphic to the disjoint union $G^{\prime}_{j_0} \cup w \cong G_{j_0}$. Hence ${\mathcal{H}} \cup v$ contains a copy of $G_{j_0}$ in color $j_0$. Since we started with an arbitrary critical coloring ${\mathcal{H}}$ of $r({\mathcal{G}})$, it follows that $r_{\ast}({\mathcal{G}}) = 0$.

($\Longrightarrow$) We assume that $r_{\ast}({\mathcal{G}}) = 0$ and $G_1, \dotsc, G_t$ are non-discrete. From Lemma \ref{bounding-lem}, we have $r({\mathcal{G}}) - 1 \leq r({\mathcal{G}}_{I_{00}}) \leq r({\mathcal{G}})$. It remains to prove that $r({\mathcal{G}}_{I_{00}}) = r({\mathcal{G}}) - 1$.
If possible assume that $r({\mathcal{G}}_{I_{00}}) = r({\mathcal{G}}) =: n$. Consider any critical coloring of ${\mathcal{G}}_{I_{00}}$; that is, a $t$-coloring ${\mathcal{Q}}$ of $K_{n-1}$ that avoids a copy of $G^{\prime}_i$ (respectively, $G_i$) in color $i$ for each $i \in I_{00}$ (respectively, $i \in [1,t] - I_{00}$). Then ${\mathcal{Q}}$ is also a critical coloring of ${\mathcal{G}}$. Introduce an isolated vertex $v$ and define ${\overline{\mathcal{Q}}} := {\mathcal{Q}} \cup v$. Since $r_{\ast}({\mathcal{G}}) = 0$, it follows that ${\overline{\mathcal{Q}}}$ contains a copy of $G_{\gamma}$ in ${\overline{\mathcal{Q}}}$ for some $\gamma \in [1,t]$ and $\delta(G_{\gamma}) = 0$. 
Now, this implies that $\gamma \in I_0$. If possible, suppose $\gamma \in I_0 - I_{00}$. Then $|V(G_{\gamma})| < n$ and replacing the vertex $v$ by a vertex from ${\mathcal{Q}}$, we have a copy of $G_{\gamma}$ in color $\gamma$ in ${\mathcal{Q}}$, a contradiction. Hence $\gamma \in I_{00}$. Now $|V(G_{\gamma})| = r({\mathcal{G}}) = n$, and still has an isolated vertex. In this case, removing the isolated vertex from $G_{\gamma}$ yields $G^{\prime}_{\gamma}$, which has a copy inside ${\mathcal{Q}}$ in color $\gamma$, again a contradiction.
\end{proof}

Theorem \ref{star-crit-zero-thm} leads to the following corollary, which is used for certain calculations in the next section.

\begin{corollary}\label{cor-lem-star-crit-zero}
Let $t \geq 2$ be an integer and ${\mathcal{G}} = \{ G_i \ | \  i \in [1,t] \}$ be a multiset of simple graphs of order at least $2$. Then we have the following.
\begin{enumerate}
\item[(i)] If $G_j$ is discrete for some $j \in [1,t]$, then 
\[
r({\mathcal{G}}) = \min \big\{ |V(G_i)| \ | \  G_i \mbox{ is discrete and } i \in [1,t] \big\}. 
\]
\item[(ii)] If $G_1, \dotsc, G_t$ are non-discrete graphs and $r({\mathcal{G}}_{I_{00}}) = r({\mathcal{G}}) - 1$, then 
\begin{align} 
r({\mathcal{G}}) &= \max \big\{ |V(G_j)| \ | \  \delta(G_j) = 0 \mbox{ and }  j \in [1,t] \big\} \notag \\
&>\max \big\{ |V(G_j)| \ | \  \delta(G_j) \geq 1 \mbox{ and } j \in [1,t] \big\}. \notag \end{align}
\end{enumerate}
\end{corollary}
 
\begin{proof}
Statement (i) is already observed in the first part of the proof of Theorem \ref{star-crit-zero-thm}. To prove Statement (ii), we first notice that $r_{\ast}({\mathcal{G}}) = 0$ from Theorem \ref{star-crit-zero-thm}. Next, we notice that $r({\mathcal{G}}) \geq |V(G_i)|$ for each $i \in [1,t]$. To see this, for an arbitrary $i \in [1,t]$, consider the $t$-coloring of $K_{|V(G_i)| - 1}$ with assigned color $i$ to every edge. Since $G_j$ is non-discrete for every $j \in [1,t]$, this avoids a copy of $G_j$ in color $j$ for every $j \in [1,t]$. 

Now we establish the inequality in Statement (ii). If possible, suppose $r({\mathcal{G}}) = |V(G_j)|$ for some $j \in [1,t]$ with $\delta(G_j) \geq 1$. Consider a $K_{|V(G_j)| - 1}$, assign the color $j$ to all its edges, and call this $t$-coloring ${\mathcal{S}}$. From the previous paragraph, it follows that ${\mathcal{S}}$ is a critical coloring of ${\mathcal{G}}$. Now, introduce a new isolated vertex $v$. Since $r_{\ast}({\mathcal{G}}) = 0$, there must be a copy $K$ of $G_{\alpha}$ in color $\alpha$ in ${\mathcal{S}} \cup v$ for some $\alpha \in [1,t]$. Since all the edges of ${\mathcal{S}} \cup v$ has color $j$ and $G_{\alpha}$ is non-discrete, this forces $\alpha = j$. Since $G_j \cong K \subseteq {\mathcal{S}} \cup v$ and $K \subsetneq {\mathcal{S}}$, this implies that $v \in V(K)$ is an isolated vertex of $K$, and consequently, $G_j$ has an isolated vertex. This contradicts $\delta(G_j) \geq 1$.   

Now we proceed to prove that $$r({\mathcal{G}}) = \max \big\{ |V(G_j)| \ | \  \delta(G_j) = 0 \mbox{ and } j \in [1,t] \big\}.$$ Since $r({\mathcal{G}})$ is an upper bound for $|V(G_i)|$ for all $i \in [1,t]$, as noted earlier, it is enough to show that the equality $r({\mathcal{G}}) = |V(G_j)|$ occurs for some $j \in [1,t]$ with $\delta(G_j) = 0$. 
Suppose this is not true (i.e., $r({\mathcal{G}}) > |V(G_j)|$ for all $j \in [1,t]$ with $\delta(G_j) = 0$). Set $n := r({\mathcal{G}})$, and consider an arbitrary critical coloring ${\mathcal{P}}$ of ${\mathcal{G}}$; i.e., a $t$-coloring of $K_{n-1}$ that avoids a copy of $G_j$ in color $j$ for all $j \in [1,t]$. Introduce a new isolated vertex $x$. Since $r_{\ast}({\mathcal{G}}) = 0$, then ${\mathcal{P}} \cup x$ contains a copy $L$ of $G_{\beta}$ for some $\beta \in [1,t]$. Since $x$ is not joined to ${\mathcal{P}}$ by any edges in ${\mathcal{P}} \cup x$, we have $\delta(G_{\beta}) = 0$. Then $x \in V(L)$, and let $w \in V(G_{\beta})$ be the vertex of $G_{\beta}$ that is identified with vertex $x$. Now $G^{\prime}_{\beta} := G_{\beta} - w \cong  L - x$. From our assumption, $|V(G^{\prime}_{\beta})| < n-1 = |V({\mathcal{P}})|$. Hence, there exists a vertex $z \in V({\mathcal{P}}) - V(L)$. But then $(L - x) \cup z \subseteq {\mathcal{P}}$ is a copy of $G_{\beta}$ in color $\beta$, giving a contradiction. \end{proof}

\section{General Lower Bounds}\label{genlower}

The first general lower bound proved for star-critical Ramsey numbers was due to Zhang, Broersma, and Chen \cite{ZBC}.  In 2016, they proved the following $2$-color lower bound.

\begin{theorem}[\label{ZhangBC}\cite{ZBC}]
Suppose that $G_1$ is a connected graph of order at least $2$ that is $G_2$-good.  If $s(G_2)=1$, $\delta (G_1)=1$, or $\kappa (G_1)\ge 2$, then $$r_*(G_1, G_2)\ge (|V(G_1)|-1)(\chi (G_2)-2)+s(G_2)+\delta (G_1)+\tau (G_2)-2.$$
\end{theorem} 

\noindent In 2018, Hao and Lin (see \cite{HL1} and \cite{HL2}) offered the following variation on Zhang, Broersma, and Chen's lower bound.

\begin{theorem}[\label{HaoLin}\cite{HL1}, \cite{HL2}]
Suppose that $G_2$ is a graph with $\chi (G_2)\ge 2$ and let $G_1$ be a connected graph satisfying $|V(G_1)|\ge s(G_2)+1$.  Then $$r_{\ast}(G_1, G_2)\ge (|V(G_1)|-1)(\chi (G_2)-2) + \min \{|V(G_1)|, \delta (G_1)+\tau (G_2)-1\}.$$  If $\kappa (G_1)\ge 2$ or $\delta (G_1)=1$, then $$r_*(G_1, G_2)\ge (|V(G_1)|-1)(\chi(G_2)-2)+\min \{|V(G_1)|, \delta (G_1)+\tau (G_2)-1\}+s(G_2)-1.$$
\end{theorem}

Theorems \ref{ZhangBC} and \ref{HaoLin} have been used in the determination of various star-critical Ramsey numbers involving cycles, fans, and other graphs.  One limitation of these lower bounds is that they are only $2$-color results.  At present, only one general lower bound is known for multicolor star-critical Ramsey numbers.  An equivalent result to the following theorem was proved by Budden and DeJonge \cite{BD} in 2022 (see also Theorem 1.3 of \cite{B}).

\begin{theorem}[\label{BudDe}\cite{BD}, \cite{B}]
Suppose that $G_1, G_2, \dots ,G_{t-1}$ are connected graphs of order at least $2$.  If $G_t$ is any graph such that $(G_1, G_2, \dots , G_{t-1})$ is $G_t$-good and $$r(G_1, G_2, \dots , G_{t-1})\ge s(G_t),$$ then \begin{align} r_*(G_1, G_2, \dots , G_{t-1}, G_t)\ge r_{\ast}(G_1, G_2, \dots , G_{t-1})&+r(G_1, G_2, \dots , G_{t-1}, G_t) \notag \\ &- r(G_1, G_2, \dots , G_{t-1}).\notag \end{align}
\end{theorem}



Before proving a generalized lower bound for star-critical Ramsey numbers, we introduce the following definition.

\begin{definition}\label{def-hyp-A}
Let $t \geq 2$ be an integer, ${\mathcal{G}} = (G_1, G_2, \dotsc, G_t)$ be a multiset of graphs.  We say that ${\mathcal{G}}$ satisfies {\it Hypothesis A} if the following four conditions are satisfied:
\begin{enumerate}
\item $|V(G_i)|>1$ for all $i\in [1,t]$, 
\item $G_i$ is connected for all $i \in [1,t-1]$,
\item $(G_1, G_2, \dots , G_{k-1})$ is $G_k$-good, for all $k\in [2, t]$, and
\item $r(G_1, G_2, \dots ,G_{k-1})\ge s(G_k)+1$, for all $k\in [2, t]$.
\end{enumerate}
\end{definition}

\noindent In Theorem \ref{star-good-zero}, we address the  exceptional case $r_{\ast}({\mathcal{G}}) = 0$, as discussed in the previous section.

\begin{theorem} \label{star-good-zero}
Let $t \geq 2$ be an integer and ${\mathcal{G}} = (G_1, G_2, \dotsc, G_t)$ be a multiset of graphs that satisfies Hypothesis A. Then the following statements are equivalent:
\begin{enumerate}
\item[(i)] $\chi(G_t) = 1$, 
\item[(ii)] $r_{\ast}({\mathcal{G}}) = 0$.  
\end{enumerate}
\end{theorem}

\begin{proof}
The statement (i)$\Longrightarrow$(ii) follows from Theorem \ref{star-crit-zero-thm}.  To prove (ii)$\Longrightarrow$(i) we assume that $r_{\ast}({\mathcal{G}}) = 0$. Now suppose (i) is not true (i.e., $\chi(G_t) \geq 2$).  If $G_t$ does not contain an isolated vertex, then $I_{00} \subseteq I_{0} = \emptyset$ using the notation in Section \ref{star-crit-zero}. From Theorem \ref{star-crit-zero-thm}, we have $r({\mathcal{G}}) = r({\mathcal{G}}_{I_{00}}) = r({\mathcal{G}}) - 1$, giving a contradiction.

Now let $w \in V(G_t)$ be an isolated vertex in $G_t$ and set $G^{\prime}_t := G_t - w$. Notice that $I_{0} = \{ t \}$ and again as argued above, we must have $I_{00} = \{ t \}$ as well. Since $\chi(G_t) \geq 2$, we have $G_1, G_2, \dotsc, G_t$ are all non-discrete graphs. Then by part (ii) of Corollary \ref{cor-lem-star-crit-zero}, it follows that $r(G_1, G_2, \dotsc, G_{t-1}, G^{\prime}_t) = r(G_1, G_2, \dotsc, G_t) - 1$.
Now, using $\chi(G_t) \geq 2$, we have $\chi(G_t) = \chi(G^{\prime}_t)$ and $s(G_t) = s(G^{\prime}_t)$ since for any proper vertex coloring of $G_t$, the set $\{ w \}$ cannot form a color class.  Next, using Hypothesis A we have
\begin{eqnarray*}
r(G_1, G_2, \dotsc, G_t) & = & (r(G_1, G_2, \dotsc, G_{t-1}) - 1)(\chi(G_t) - 1) + s(G_t) \\
	& = & (r(G_1, G_2, \dotsc, G_{t-1}) - 1)(\chi(G^{\prime}_t) - 1) + s(G^{\prime}_t)
\end{eqnarray*}
which implies that
\[
r(G_1, G_2, \dotsc, G_{t-1}, G^{\prime}_t) = (r(G_1, G_2, \dotsc, G_{t-1}) - 1)(\chi(G^{\prime}_t) - 1) + s(G^{\prime}_t) - 1.
\]
and the multiset $(G_1, G_2, \dotsc, G_{t-1}, G^{\prime}_t)$ also satisfies the hypothesis required for Inequality (\ref{burr-gen-ineq}). Then we also have
\[
r(G_1, G_2, \dotsc, G_{t-1}, G^{\prime}_t) \geq (r(G_1, G_2, \dotsc, G_{t-1}) - 1)(\chi(G^{\prime}_t) - 1) + s(G^{\prime}_t),
\]
giving a contradiction. 
\end{proof}

Our next goal is to establish a general multicolor lower bound for star-critical Ramsey numbers. Our new bound both extends Theorems \ref{ZhangBC} and \ref{HaoLin} to the multicolor setting and improves the bound given in Theorem \ref{BudDe} for certain collections of graphs. In order to state it, we must iterate the process of introducing graphs to Ramsey and star-critical Ramsey numbers.

\begin{definition}
For a $t$-tuple of graphs $(G_1, G_2, \dots , G_t)$, where $t\ge 2$, define the {\it characteristic sequence} $\{ d_k\}_{k=1}^t$ as follows.  First let $d_1=\delta (G_1)-1$.  Then for each $k\in [2,t]$, let $R_{k-1}=r(G_1, G_2, \dots , G_{k-1})$ and define 
\begin{equation*}
d_k=\left\{\begin{aligned} 
(R_{k-1}-1)&(\chi (G_k)-2)+d_{k-1}+s(G_k)+\tau (G_k)-2  \notag \\
&\mbox{if $s(G_k)=1$, $d_{k-1}=0$, or $\kappa (G_i)\ge 2$ for all $i \in [1,k-1]$} \notag \\
(R_{k-1}-1)&(\chi(G_k)-2)+\min \{R_{k-1}, \tau(G_k)+d_{k-1}\}-1 \notag \\
&\mbox{if $1\le s(G_k)-1\le d_{k-1}$ and $\kappa (G_i)=1$ for some $i \in [1,k-1]$} \notag \\
(R_{k-1}-1)&(\chi(G_k)-2)+s(G_k)+\tau(G_k)-2 \notag \\
&\mbox{if $1\le d_{k-1}<s(G_k)-1$ and $\kappa (G_i)=1$ for some $i \in [1,k-1]$}. \notag 
\end{aligned} \right.
\end{equation*}
\end{definition}

\begin{theorem}\label{lower}
Let $t \geq 2$ be an integer and ${\mathcal{G}} = (G_1, G_2, \dotsc, G_t)$ be a multiset of graphs that satisfies Hypothesis A. Assume that $\chi(G_t) \geq 2$. 
Then $$r_*(G_1, G_2, \dots , G_k)\ge d_k+1,$$ for all $k\in [2,t]$.
\end{theorem}

\begin{proof}
We proceed by induction on $k\ge 1$.  In the base case $k=1$, we trivially have $r_*(G_1)=\delta (G_1)=d_1+1$.  Now let $2\le k\le t$ and assume that $r_*(G_1, G_2, \dots , G_{k-1})\ge d_{k-1}+1$.  Then there exists a $(k-1)$-coloring of $$K_{r(G_1, G_2, \dots, G_{k-1})-1}\sqcup K_{1, d_{k-1}}$$ that avoids a copy of $G_i$ in color $i$, for all $i\in [1, k-1]$.  Denote this $(k-1)$-colored graph by $\mathcal{G}_1$ and let $w$ be the vertex of degree $d_{k-1}$.  Denote by $\mathcal{G}_2$ the $(k-1)$-colored complete graph formed by deleting vertex $w$ from $\mathcal{G}_1$. Let $D_0 \subset V(\mathcal{G}_2)$ comprise of the $d_{k-1}$ vertices in $\mathcal{G}_1$ that are adjacent to $w$ (note that $D_0$ is a proper subset of $V(\mathcal{G}_2)$).  

Let $\mathcal{G}_3$ be a subgraph of $\mathcal{G}_2$ induced on any subset of  $s(G_k)-1$ vertices in $V(\mathcal{G}_2)$. As $|\mathcal{G}_3|= s(G_k)-1 \leq r(G_1, G_2, \dotsc, G_{k-1})-2= |\mathcal{G}_2|-1$, there exists a vertex $w' \in V(\mathcal{G}_2)-V(\mathcal{G}_3)$. Let $\mathcal{G}_2[V(\mathcal{G}_3)\cup w']$ be the subgraph of $\mathcal{G}_2$ which is induced by the vertices of $\mathcal{G}_3$ along with $w'$. Here, $\mathcal{G}_3 \subseteq \mathcal{G}_2[V(\mathcal{G}_3)\cup w'] \subseteq \mathcal{G}_2 \subseteq \mathcal{G}_1$ and all of them avoid a copy of $G_i$ in color $i$, for all $i\in [1, k-1]$.

From Assumption (3) in the statement of Hypothesis A (Definition \ref{def-hyp-A}), it follows that $$r(G_1, G_2, \dots, G_k)=(r(G_1, G_2, \dots , G_{k-1})-1)(\chi (G_k)-1)+s(G_k).$$  In order to construct a critical coloring for $(G_1, G_2, \dots , G_k)$, begin with a copy of $K_{\chi (G_k)}$ in color $k$ and replace one of its vertices with a copy of $\mathcal{G}_3$ (name this vertex set $X_{\chi(G_k)}$) and the other $\chi (G_k)-1$ of its vertices are replaced with copies of $\mathcal{G}_2$ (name these vertex sets $X_1, X_2, \dots, X_{\chi(G_k)-1}$).  The resulting $$K_{(r(G_1, G_2, \dots , G_{k-1})-1)(\chi (G_k)-1)+s(G_k)-1}$$ avoids a monochromatic copy of $G_i$ in color $i$, for all $i\in [1, k-1]$ since every $G_i$ is assumed to be connected.  To see that it also avoids a copy of $G_k$ in color $k$, consider two cases.  First, if $s(G_k)=1$, then coloring the vertices according to which copy of $K_{r(G_1, G_2, \dots , G_{k-1})-1}$ they are in leads to a proper vertex coloring of the subgraph spanned by edges in color $k$ that uses $\chi (G_k)-1$ colors.  If $s(G_k)>1$, then we obtain a proper vertex coloring for the subgraph spanned by edges in color $k$ that uses $\chi (G_k)$ colors, but has a color class with only $s(G_k)-1$ colors.  In both cases, we see that a copy $G_k$ in color $k$ does not exist.  Hence, we have produced a $k$-coloring of $K_{r(G_1, G_2, \dots , G_k)-1}$ that avoids a copy of $G_i$ in color $i$, for all $i\in [1,k]$.
Introduce a vertex $v$ to this critical coloring above, which will be the centre of the star. We divide the remainder of the proof into cases.

\underline{\bf Case 1} Assume that $s(G_k)=1$, $d_{k-1}=0$, or $\kappa (G_i)\ge 2$ for all $1\le i\le k-1$. 
Join $v$ to vertices in $X_1, X_2, \dots, X_{\chi(G_k)-2}$ by edges in color $k$.
Then $X_{\chi(G_k)-1} \cong \mathcal{G}_2$, and under this isomorphism let $D \subset X_{\chi(G_k)-1}$  correspond to $D_0$ and vertex $z_i \in D$ correpsond to vertex $y_i \in D_0$, for all $i \in [1,d_{k-1}]$. Join $v$ to the vertices in $D$, coloring edge $vz_i$ the same color as edge $wy_i$, for all $i \in [1,d_{k-1}]$. Similar to that of $\mathcal{G}_1$, this coloring avoids a copy of $G_j$ in color $j$, for all $j \in [1,k-1]$.

Also, $X_{\chi(G_k)} \cong \mathcal{G}_3$, and under this isomorphism, let vertex $h_i \in X_{\chi(G_k)}$ correspond to vertex $g_i \in V(\mathcal{G}_3)$, for all $i\in [1, s(G_k)-1]$. Join $v$ to the vertices in $X_{\chi(G_k)}$, coloring edge $vh_i$ the same color as edge $w'g_i$, for all $i\in [1, s(G_k)-1]$. Similar to that of $\mathcal{G}_2[V(\mathcal{G}_3)\cup w']$, this coloring avoids a copy of $G_j$ in color $j$, for all $j \in [1, k-1]$.

Let $T \subseteq X_{\chi(G_k)-1}-D$ be chosen such that $$|T|= \min \{r(G_1, G_2, \dotsc, G_{k-1})-1-d_{k-1}, \tau(G_k)-1\}.$$ Join $v$ to the vertices in $T$ by edges in color $k$.
Call this newly constructed graph $L$ and note that there is no copy of $G_k$ in color $k$ in $L$. Denote by $L_k$ the subgraph of $L$ spanned by edges in color $k$.

Color the vertices in $X_i$ by color $c_i$ such that $c_i \neq c_j$, for all $i,j \in [1, \chi(G_k)]$ and $i \neq j$. As $v$ is not joined to $X_{\chi(G_k)}$ by any edges of color $k$, assign color $c_{\chi(G_k)}$ to $v$. This is a proper vertex coloring of $L_k$ using $\chi(G_k)$ colors.  As $X_i$ is joined to $X_j$ by edges of color $k$ for all $i,j \in [1, \chi(G_k)]$ and $i \neq j$, this is the only possible vertex coloring of $L_k$ using $\chi(G_k)$ colors (up to a permutation of the colors). Corresponding to this vertex coloring, $X_{\chi(G_k)} \cup v$ is the smallest color class and stays so, for all proper vertex colorings of $L_k$. Thus, $$s(L_k)=|X_{\chi(G_k)} \cup v|=s(G_k)-1+1=s(G_k).$$ For each vertex $w_i \in X_{\chi(G_k)}$, $w_i$ is connected to each of the other color classes by $r(G_1, G_2, \dotsc, G_{k-1})-1$ many edges. However, $v$ is connected to $X_{\chi(G_k)-1}$ by $$|T|=\min \{r(G_1, G_2, \dotsc, G_{k-1})-1-d_{k-1}, \tau(G_k)-1\}$$ many edges. 
Thus, $\tau(L_k) = |T| \leq \tau(G_k)-1$ implying $G_k \nsubseteq L_k$.
So, there is no copy of $G_i$ in color $i$ in $L$ for all $i\in [1, k-1]$.
 
If $s(G_k)=1$, then $X_{\chi(G_k)}=\emptyset$. Hence, the $\{1, 2, \dots, k-1\}$ (edge) colored connected components have vertex sets $X_1, X_2, \dotsc, X_{\chi(G_k)-2}, X_{\chi(G_k)-1}\cup v$ individually, which are nothing but copies of $\mathcal{G}_2$ and $\mathcal{G}_1$. Hence, they avoid a copy of $G_i$ in color $i$ for all $i\in [1, k-1]$.

If $d_{k-1}=0$, then $D=\emptyset$ and $v$ is not joined to vertices in $X_{\chi(G_k)-1}$ by any edge of color ${1,2, \dotsc, k-1}$. Hence, the only $\{1, 2, \dots, k-1\}$ (edge) colored connected components have vertex sets $X_1, X_2, \dotsc, X_{\chi(G_k)-2}, X_{\chi(G_k)-1}, X_{\chi(G_k)}\cup v$, which are nothing but copies of $\mathcal{G}_2$ and $\mathcal{G}_2[V(\mathcal{G}_3)\cup w']$. Hence, they avoid a copy of $G_i$ in color $i$ for all $i\in [1, k-1]$.

If $s(G_k)>1$ and $d_{k-1}\ge 1$, then assume that $\kappa(G_i)\ge 2$ for all $i\in [1, k-1]$. If $L$ contains a copy of $G_i$ in color $i$ for some $i\in [1, k-1]$, then from the above two cases it is clear that the copy of $G_i$ must have at least one of it's vertices in $X_{\chi(G_k)}$ and at least one vertex in $X_{\chi(G_k)-1}$, implying $v$ to also be a vertex in this copy of $G_i$. But $v$ is a cut-vertex, implying $\kappa(G_i)=1$, a contradiction. Hence, there is no copy of $G_i$ in color $i$ in $L$ for all $i\in [1, k-1]$.

Thus, $L$ avoids a copy of $G_j$ in color $j$ for all $j\in [1, k]$. If $ \tau(G_k)-1 \geq r(G_1, G_2, \dotsc, G_{k-1}) - 1 - d_{k-1}$ then $v$ is connected to all the vertices in $K_{r(G_1,\dotsc,G_k)-1}$ and we have a $k$-coloring of $K_{r(G_1, G_2, \dotsc, G_k)}$. So this must contain a copy of $G_j$ in color $j$, for some $j\in [1, k]$, giving a contradiction. 

Hence, $\tau(G_k)-1 < r(G_1, G_2, \dotsc, G_{k-1})-1-d_{k-1}$ and $|T|=\tau(G_k)-1$, from which it follows that
\begin{equation*}
    \begin{split}
        r_*(G_1, G_2, \dotsc, G_k) \geq (r(G_1, G_2, \dotsc, G_{k-1})-1)(\chi(G_k)-2) + s(G_k)-1\\ + d_{k-1} + \tau(G_k)-1 +1.
    \end{split}
\end{equation*}
Thus, $r_*(G_1, G_2, \dotsc, G_k) \geq d_k +1$.

\underline{\bf Case 2} Assume that $1\le s(G_k)-1\le d_{k-1}$ and $\kappa (G_i)=1$ for some $i\in [1, k-1]$. 
In this case, remove the edges joining $v$ to $X_{\chi(G_k)}$ in L.
As we are not changing the construction of $k$ colored edges, this new graph does not contain a copy of $G_k$ in color $k$. And all the $\{1, 2, \dots, k-1\}$ (edge) colored connected components have vertex sets $X_1, X_2, \dotsc, X_{\chi(G_k)-2}, X_{\chi(G_k)-1}\cup v$ individually, which are just copies of $\mathcal{G}_2$ and $\mathcal{G}_1$. Hence, there is no copy of $G_j$ in color $j$ for all $j\in [1, k]$. Thus,
\begin{equation*}
    \begin{split}
        r_*(G_1, G_2, \dotsc, G_k) & \geq (r(G_1, G_2, \dotsc, G_{k-1})-1)(\chi(G_k)-2) + d_{k-1} + |T| +1\\
        & = (r(G_1, G_2, \dotsc, G_{k-1})-1)(\chi(G_k)-2) + d_{k-1}\\
        & \hspace{0.5cm} + \min \{r(G_1, G_2, \dotsc, G_{k-1})-1-d_{k-1}, \tau(G_k)-1\} +1\\
        & = (r(G_1, G_2, \dotsc, G_{k-1})-1)(\chi(G_k)-2)\\ 
        & \hspace{0.5cm} + \min \{r(G_1, G_2, \dotsc, G_{k-1}), \tau(G_k)+ d_{k-1}\}-1+1.
    \end{split}
\end{equation*}
Thus, $r_*(G_1, G_2, \dotsc, G_k) \geq d_k +1$.

\underline{\bf Case 3} Assume that $1\le d_{k-1}<s(G_k)-1$ and $\kappa (G_i)=1$ for some $i \in [1, k-1]$. 
In this case, remove the edges joining $v$ to $X_{\chi(G_k)-1}$ in $L$.
As $v$ is not joined to vertices in $X_{\chi(G_k)-1}$ by edges of color ${1,2, \dotsc, k-1}$, the only $\{1, 2, \dots, k-1\}$ (edge) colored connected components have vertex sets $X_1, \dots, X_{\chi(G_k)-1}$ and $X_{\chi(G_k)}\cup v$ individually, which are just copies of $\mathcal{G}_2$ and $\mathcal{G}_2[V(\mathcal{G}_3)\cup w']$.

Here, instead of $T$, let $T_1 \subseteq X_{\chi(G_k)-1}$ such that $|T_1|= \min \{r(G_1, G_2, \dotsc, G_{k-1})-1, \tau(G_k)-1\}$. Join $v$ to vertices in $T_1$ by edges of color $k$. Call this newly constructed graph $\mathcal{L}$. Following the similar vertex coloring argument using $\chi(G_k)$ many colors as in Case I, we can conclude that $\tau(\mathcal{L}_k) \leq \tau(G_k)-1$ implying that there is no copy of $G_k$ in color $k$ in this construction.
Thus, there is no copy of $G_j$ in color $j$ for all $j\in [1, k]$.

If $r(G_1, G_2, \dotsc, G_{k-1})-1 \leq \tau(G_k)-1$ then $v$ is connected to all the vertices in $K_{r(G_1, G_2, \dotsc, G_k)-1}$ and we have a $k$-coloring of $K_{r(G_1, G_2, \dotsc, G_k)}$.   This coloring necessarily contains a copy of $G_j$ in color $j$, for some $j\in [1, k]$, giving a contradiction. Hence, $\tau(G_k)-1 < r(G_1, G_2, \dotsc, G_{k-1})-1$ and $|T_1|=\tau(G_k)-1$.
Thus,
\begin{equation*}
    \begin{split}
        r_*(G_1, G_2, \dotsc, G_k) \geq (r(G_1, G_2, \dotsc, G_{k-1})-1)(\chi(G_k)-2)+s(G_k)-1\\+\tau(G_k)-1 +1,
    \end{split}
\end{equation*}
from which it follows that $r_*(G_1, G_2, \dotsc, G_k) \geq d_{k}+1$
\end{proof}

\section{Some Star-Critical Ramsey Numbers for Paths}\label{3paths}

In 1967, Gerencs\'er and Gy\'arf\'as \cite{GG} proved that if $m\ge n\ge 2$, then $$r(P_m, P_n)=m+\floor{\frac{n}{2}}-1.$$ Hook \cite{H} considered the star-critical analogue of this number by showing that $r_*(P_m, P_n)=\ceil{\frac{n}{2}}$, for all $m\ge n\ge 2$.  In the case of three colors, Maherani, Omidi, Raeisi, and Shahsiah \cite{MORS} proved that for all $m\ge n\ge 3$ and $(m,n)\ne (3,3), (4,3)$, \begin{equation}r(P_m, P_n, P_3)=m+\floor{\frac{n}{2}}-1.\label{rampath}\end{equation}  Note that $(P_m,P_n)$ is $P_3$-good for these values of $m$ and $n$.  The cases $$r(P_3, P_3, P_3)=5=r(P_4, P_3, P_3)$$ can be found in \cite{AKM}.  

Theorem \ref{lower} implies that
if $m\ge n\ge 3$ and $(m,n)\ne (3,3), (4,3)$, then \begin{equation}r_*(P_m, P_n, P_3)\ge \ceil{\frac{n}{2}}+1.\label{3colorlower}\end{equation}
We also need the following well-known factorization theorem.  Recall that a {\it $1$-factor} of a graph $G$ is an independent set of edges that span $G$.  A graph is said to have a {\it $1$-factorization} if its edge set is the disjoint union of $1$-factors.

\begin{theorem}[\label{factor}\cite{Harary}]
If $n$ is even, then the complete graph $K_{n}$ has a $1$-factorization.
\end{theorem}

Using the main result of \cite{FLPS}, we have $r(C_5,P_3)=5$ and $r(C_4,P_3)=4$. In the following theorem we determine $r_{\ast}(C_5,P_3)$, which we will be needed for the main result of this section.

\begin{theorem}\label{star-c5p3}
    $r_{\ast}(C_5,P_3) = 3$.
\end{theorem}

\begin{proof}
Theorem \ref{ZhangBC} implies that $r_{\ast}(C_5,P_3) \ge 3$. To prove the reverse inequality, consider a $2$-coloring of $K_4 \sqcup K_{1,3}$ using red and blue and let $w$ be the centre vertex of the star.  Since $r(C_4, P_3)=4$, if the $K_4$-subgraph does not contain a blue $P_3$, then it must contain a red $C_4$, which we assume is given by $abcda$.  
As we are to avoid a blue $P_3$, at least $2$ red edges are adjacent to $w$.

If $w$ is adjacent with red edges to two vertices in the red $C_4$ (say $\{a,d\}$), then $wabcdw$ is a red $C_5$ (see image (i) in Figure \ref{c5p3cases}). 
    \begin{figure}[h!]
        \centerline{
        \includegraphics[width=0.9\textwidth]{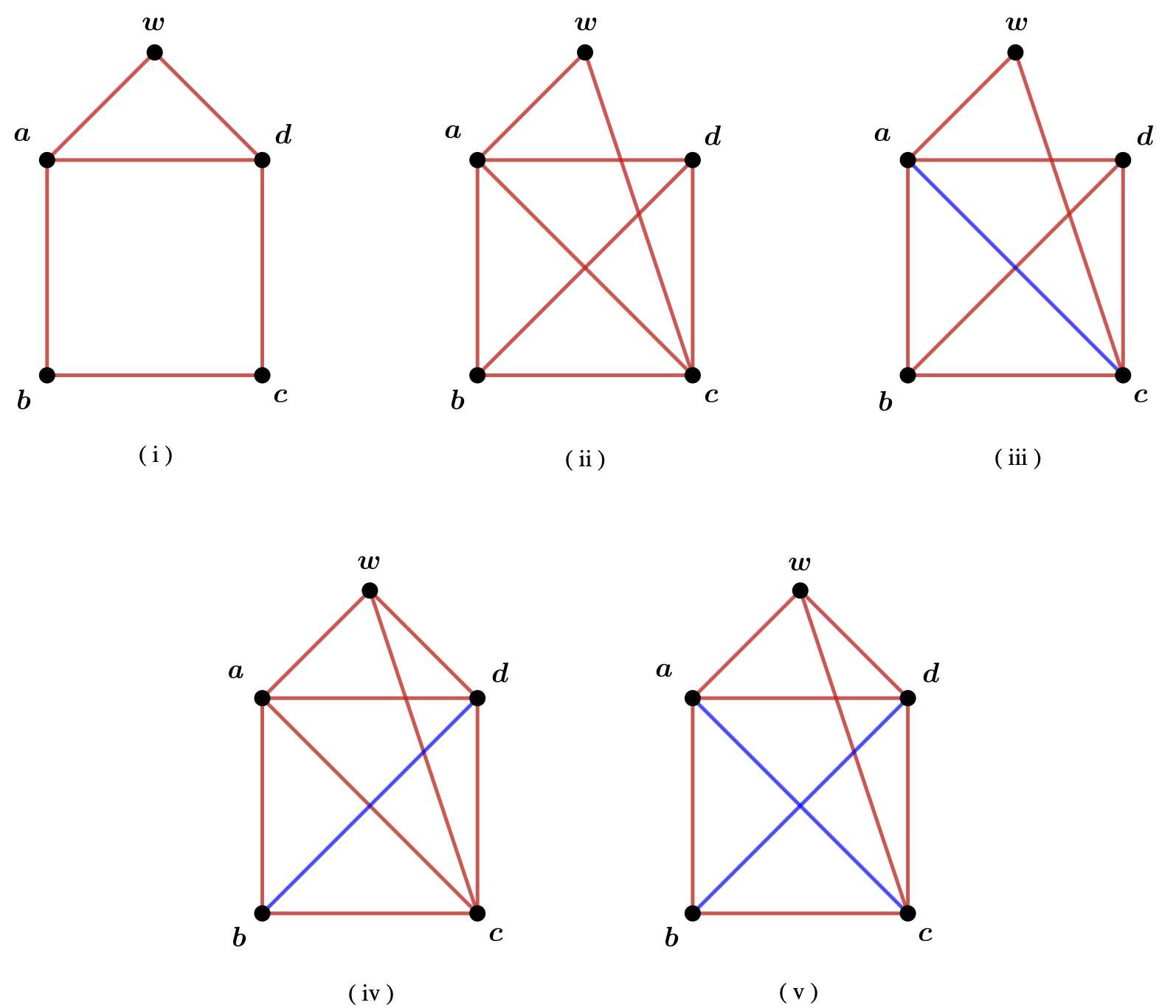} }
        \caption{$2$-colorings of $K_4\sqcup K_{1,3}$ that shows $r_*(C_5,P_3) \leq 3$}\label{c5p3cases}
    \end{figure} 
If $w$ is adjacent with red edges to two non-adjacent diagonal vertices of the red $C_4$ (say $\{a,c\}$), then we consider the colors of $ac$ and $bd$.

If both of $ac$ and $bd$ are red, then $wabdcw$ is a red $C_5$ (see image (ii) in Figure \ref{c5p3cases}). The same argument holds if $ac$ is blue and $bd$ is red (see image (iii) in Figure \ref{c5p3cases}).
If $ac$ is red and $bd$ is blue, then $w$ must also be adjacent to either $b$ or $d$ by a red edge (say $d$), then $wabcdw$ is a red $C_5$ (see image (iv) in Figure \ref{c5p3cases}). The same argument holds if $ac$ and $bd$ are both are blue (see image (v) in Figure \ref{c5p3cases}).  In all cases, we find that there is a red $C_5$ or a blue $P_3$, from which it follows that $r_*(C_5, P_3)\le 3$.
\end{proof}

\begin{theorem}\label{k33theorem}
If $k\ge 3$, then $$r_*(P_k, P_3, P_3)=\left\{ \begin{array}{ll} 1 & \mbox{if $k=3$} \\ 3 & \mbox{if $k=4$} \\ 4 & \mbox{if $k=5$} \\ 3 & \mbox{if $k\ge 6$.}\end{array}\right.$$
\end{theorem}

\begin{proof}
We break the proof up into the cases indicated in the statement of the theorem.

\underline{\bf Case 1} Assume that $k=3$ so that $r(P_3, P_3, P_3)=5$.  Consider a $3$-coloring of $K_4\sqcup K_{1,1}$ and let vertex $x$ be the unique vertex of degree $4$.  By the Pigeonhole Principle, at least two of the edges incident with $x$ must be the same color, forming a monochromatic $P_3$.  It follows that $r_*(P_3, P_3, P_3)=1$.

\underline{\bf Case 2} Assume that $k=4$ so that $r(P_4, P_3, P_3)=5$.  By Theorem \ref{factor}, the complete graph $K_4$ can be factored into three $1$-factors.  Color each $1$-factor with a unique color from red, blue, and green with edges $ad$ and $bc$ being red.  Introduce vertex $v$ and the red edges $va$ and $vd$.  The resulting $K_4\sqcup K_{1,2}$ avoids a red $P_4$, a blue $P_3$, and a green $P_3$ (see Figure \ref{P4P3P3case}).  
\begin{figure}[h!]
\centerline{
\includegraphics[width=0.30\textwidth]{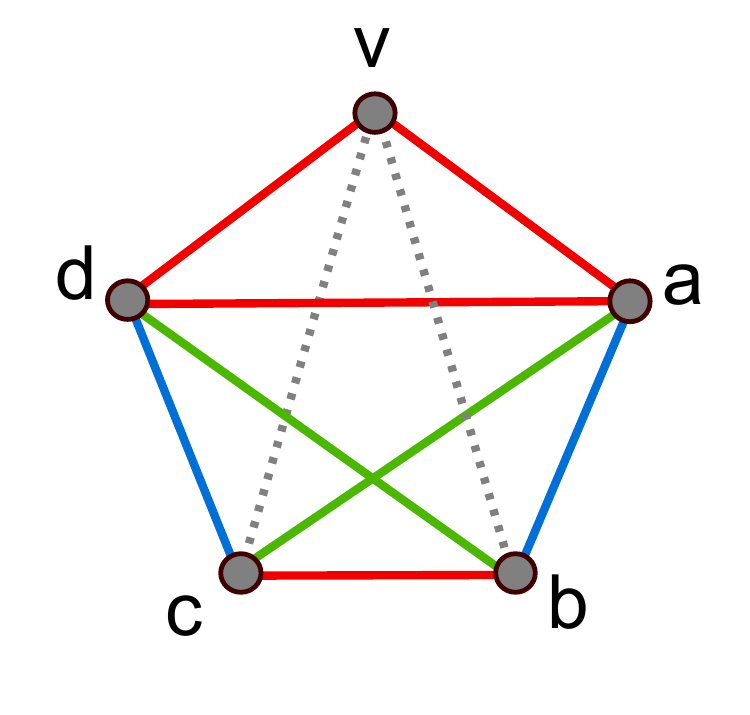} }
\caption{A $3$-coloring of $K_4\sqcup K_{1,2}$ that avoids a red $P_4$, a blue $P_3$, and a green $P_3$.}\label{P4P3P3case}
\end{figure} 
It follows that $r_*(P_4, P_3, P_3)\ge 3$.  In the process of proving the reverse inequality, we will show that this coloring is unique if a red $P_4$, a blue $P_3$, and a green $P_3$ are to be avoided. 

Now consider a $3$-coloring of $K_4\sqcup K_{1,3}$ (using red, blue, and green) and let $v$ be the center vertex of the star.  If any color is missing in the $K_4$, then the Ramsey numbers $r(P_3, P_3)=3$ and $r(P_4, P_3)=4$ imply that there is a red $P_4$, a blue $P_3$, or a green $P_3$. So, every color must appear in the $K_4$.  Denote the vertex set for the $K_4$ by $\{a,b,c,d\}$ and assume that $ab$ is blue. If $cd$ is red or green, then no other blue edges exist in the $K_4$ without forming a blue $P_3$.  Since $$r(P_4, P_3)=4 \quad \mbox{and} \quad r_*(P_4, P_3)=2,$$ it follows that there exists a red $P_4$ or a green $P_3$.  Hence, $cd$ must be blue.  If any three of the edges $ac$, $ad$, $bc$, and $bd$ are red, then they form a red $P_4$ (e.g., if $ac$, $ad$, and $bc$ are red, then $bcad$ is a red $P_4$).  So, exactly two of these edges must be green and they must be disjoint.  Without loss of generality, assume that $ac$ and $bd$ are green and $ad$ and $bc$ are red.  Now introduce vertex $v$, joining it to the $K_4$ with three edges.  If any such edge is blue or green, then a blue or green $P_3$ is formed.  So, all three edges joining $v$ to the $K_4$ are red, and since one such edge must join $v$ to $\{a,d\}$ and one must join $v$ to $\{b, c\}$, we obtain a red $P_4$ (e.g., if $va$ and $vb$ are red, then $davb$ is a red $P_4$).  It follows that $r_*(P_4, P_3)\le 3$.

\underline{\bf Case 3} Assume that $k=5$ so that by Equation (\ref{rampath}), $r(P_5, P_3, P_3)=5$.  Then the $3$-coloring of $K_4\sqcup K_{1,3}$ given in Figure \ref{P5P3P3case} implies that $r_*(P_5, P_3, P_3)=4$.
\begin{figure}[h!]
\centerline{
\includegraphics[width=0.30\textwidth]{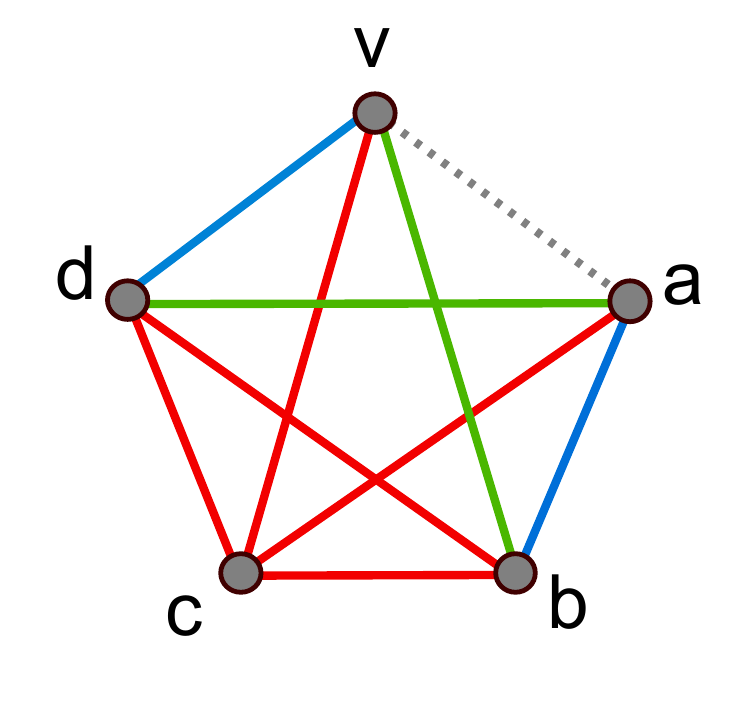} }
\caption{A $3$-coloring of $K_4\sqcup K_{1,3}$ that avoids a red $P_5$, a blue $P_3$, and a green $P_3$.}\label{P5P3P3case}
\end{figure} 

\underline{\bf Case 4} Assume that $k=6$ so that by Equation (\ref{rampath}), $r(P_6, P_3, P_3)=6$.  Inequality (\ref{3colorlower}) implies that $r_*(P_6, P_3, P_3)\ge 3$.
Consider a $3$-colored $K_5 \sqcup K_{1,3}$ (using red, blue, and green) that avoids a blue $P_3$ and a green $P_3$. 
Let $v$ be the center vertex of the star and $H$ be the underlying $K_5$ subgraph.   Since $r(P_5, P_3, P_3)=5$ by Equation (\ref{rampath}), $H$ contains a red $P_5$.
If all the edges incident with $v$ are red, then there exists a red $P_6$.  As we are to avoid a blue $P_3$ and a green $P_3$, at least one edge adjacent to $v$ must be red. If $H$ has at most one blue edge, 
then $r(C_5,P_3)=5$ and $r_{\ast}(C_5,P_3)=3$ (from \cite{FLPS} and Theorem \ref{star-c5p3}) imply that $H$ must contain a red $C_5$ (assume it is given by $x_1x_2x_3x_4x_5x_1$). As $v$ is joined to the $K_5$ subgraph by at least one red edge, without loss of generality let $vx_1$ be red. Then $vx_1x_2x_3x_4x_5$ is a red $P_6$. This implies that all colors are present in the $K_5 \sqcup K_{1,3}$ and the blue and the green subgraphs of $H$ are both matchings of size $2$. Hence, we consider two subcases where the subgraph of $H$ spanned by the green and blue edges has order $4$ or $5$, as shown in Figure \ref{P6P3P3subcases}.

\begin{figure}[h!]
\centerline{
\includegraphics[width=0.25\textwidth]{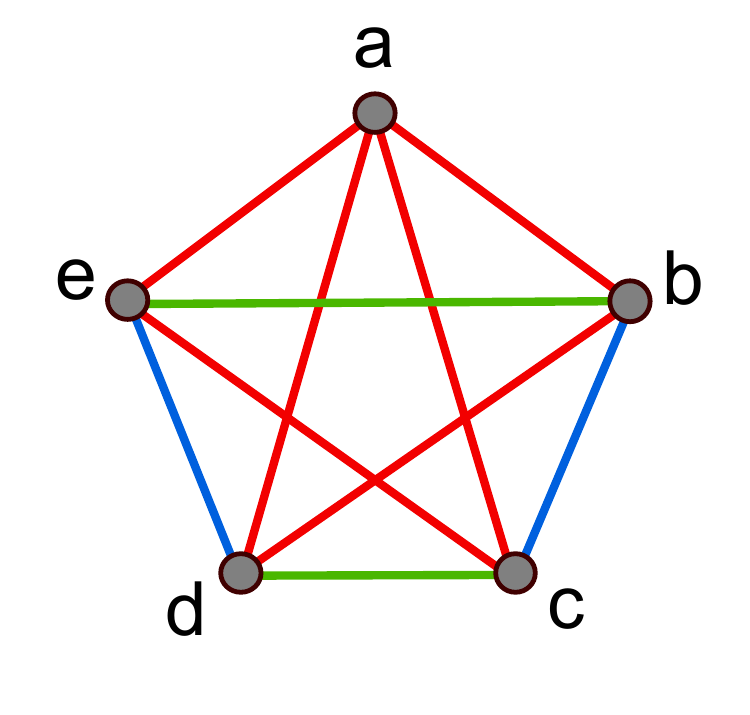} \qquad \includegraphics[width=0.25\textwidth]{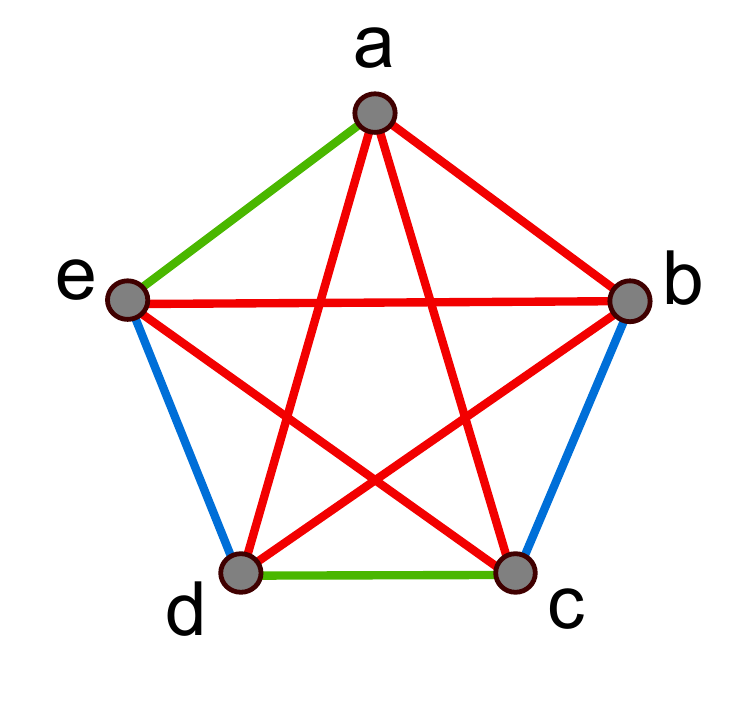}    }
\caption{Two critical colorings for $r(P_6, P_3, P_3)$ corresponding to Case 4 in the proof of Theorem \ref{k33theorem}.}\label{P6P3P3subcases}
\end{figure} 

\smallskip

\underline{Subcase 4.1} Suppose that the subgraph of $H$ spanned by the blue and green edges has order $4$.  Up to isomorphism, $H$ must be colored as in the first image in Figure \ref{P6P3P3subcases}, with the vertices labelled by $a,b,c,d,e$, as shown. Since $v$ joins to $H$ with three edges, at least two of those edges must join to the vertices in $\{b,c,d,e\}$. Without loss of generality, these edges are $(vb,vc)$ or $(vb,vd)$. If either of these edges are blue (resp. green) then a blue (resp. green) $P_3$ is formed, so these must be red. If these edges are $vb$ and $vc$ (resp. $vb$ and $vd$), then $adbvce$ (resp. $ceabvd$) is a red $P_6$.

\underline{Subcase 4.2} Suppose that the subgraph of $H$ spanned by the blue and green edges has order $5$.  Up to isomorphism, $H$ must be colored as in the second image in Figure \ref{P6P3P3subcases}. Note that $cadbec$ is a red $C_5$. Since $v$ joins to $H$ with at least one red edge, without loss of generality assume this to be $va$ or $vc$. Then either $vadbec$, or $vcadbe$ is a red $P_6$.

\smallskip

\underline{\bf Case 5} Assume that $k\ge 7$ so that $r(P_k, P_3, P_3)=k$.  Inequality (\ref{3colorlower}) implies that $r_*(P_k, P_3, P_3)\ge 3$.  Now consider a red-blue-green coloring of $K_{k-1}\sqcup K_{1,3}$ and let $v$ be the center vertex of the missing star.  Since $r(C_{k-1}, P_3, P_3)=k-1$ (see \cite{Dz}), if the subgraph $K_{k-1}$ avoids a blue $P_3$ and a green $P_3$, then it must contain a red $C_{k-1}$ (assume it is given by the vertices $a_1, a_2, \dots, a_{k-1}, a_1$ in this order).  Since $v$ joins to the $K_{k-1}$ via three edges, at least one such edge must be red if a blue $P_3$ and a green $P_3$ are avoided.  Without loss of generality, assume that $va_1$ is red.  Then $va_1a_2\cdots a_{k-1}$ is a red $P_k$.
\end{proof}

\section{Conclusion} 

In the beginning of Section \ref{3paths}, it was noted that $(P_m, P_3)$ is $P_3$-good for all $m\ge 5$.  This observation, along with the known values $r(P_m, P_3, P_3)=m=r(P_m, P_3)$ and $r_*(P_m, P_3)=2$, for all $m\ge 5$, allow us to apply Theorem \ref{BudDe} to obtain the lower bound $r_*(P_m, P_3, P_3)\ge 2$.  Of course, we saw in the previous section that Theorem \ref{lower} offers the improved lower bound $r_*(P_m, P_3, P_3)\ge 3$.
What other multicolor star-critical Ramsey numbers have improved lower bounds from using Theorem \ref{lower}?

\bibliographystyle{amsplain}

\end{document}